\documentclass{amsart}

\usepackage{color,amssymb,amsfonts,amsmath,amsthm,amscd}
\usepackage{pdfsync}

\usepackage{hyperref}
\hypersetup{
    colorlinks=true,       
    linkcolor=blue,          
    citecolor=blue,        
    filecolor=blue,      
    urlcolor=blue           
}

\usepackage[arrow,all,cmtip,matrix, curve]{xy}

\newcommand{\G}[2]{\ensuremath{\mathbb{G}(#1, #2)}}
\newcommand{\HC}[4]{\ensuremath{H_{#1}C^{#2}(#4_1,...,#4_{#3})}}
 \newcommand{\C}[3]{\ensuremath{C^{#1}(#3_1,...,#3_{#2})}}
\newcommand{\PGL}[2]{\ensuremath{\mathrm{PGL}(#1, #2)}}

\newcommand{\mult}[2]{\ensuremath{\mathrm{mult_{\mathit{p}}}(#1\cap #2)}}

\newcommand{\osc}[2]{\ensuremath{\mathit{P}(#1, #2)}}

\newtheorem{theorem}{Theorem}
\newtheorem{lemma}[theorem]{Lemma}
\newtheorem{proposition}[theorem]{Proposition}
\newtheorem{example}[theorem]{Example}

\newtheorem{corollary}[theorem]{Corollary}
\newtheorem{remark}[theorem]{Remark}

\begin{document}
\title[Weierstrass weight of the hyperosculating points]{Weierstrass weight of the hyperosculating points of generalized Fermat curves}

\author[R. A. Hidalgo]{Rub\'en A. Hidalgo}
\address{Departamento de Matem\'atica y Estad\'istica, Universidad de la Frontera, Temuco, Chile}
\email{ruben.hidalgo@ufrontera.cl}

\author[M. Leyton-\'Alvarez]{Maximiliano Leyton-\'Alvarez}
\address{Instituto de Matem\'atica y F\'isica, Universidad de Talca, Talca, Chile}

\email{leyton@inst-mat.utalca.cl}

\thanks{The first author was partially supported by the projects
 Fondecyt 1150003 and Anillo ACT1415 PIA-CONICYT.
The second author was partially supported by project Fondecyt 1170743}
\maketitle


\begin{abstract}
Let $(S,H)$ be a generalized Fermat pair of the type $(k,n)$. If $F\subset S$ 
is the set of  fixed points of the non-trivial elements of the group $H$, then $F$ is exactly 
the set of hyperosculating points of the standard embedding $S\hookrightarrow {\mathbb{P}}^{n}$. We provide an  optimal lower bound (this being sharp in
a dense open set  of  the moduli space of the generalized Fermat curves) for the Weierstrass weight of these points.
\end{abstract}

\section{Introduction}
The geometry of hyperbolic closed Riemann surfaces can be described via different objects: complex 
projective algebraic curves,  Fuchsian and Schottky groups, Jacobian varieties, etc.
Also, important to each  Riemann surface is its group of conformal automorphisms as those surfaces with a
 non-trivial group of conformal automorphisms determine the branch locus in the moduli space.
In genus at least four the branch locus coincides with its topological singular locus \cite{Rauch}.
Correspondence between these different descriptions are, in the general situation, only existential; most of the known classic examples for which this is explicitly done are rigid, that is, they have no moduli (they correspond to those Riemann surfaces with a large group of conformal automorphisms).  Some correspondences are also known for the case of cyclic $n$-gonal curves (including the case of hyperelliptic Riemann surfaces).  Generating families of Riemann surfaces  where these different descriptions are concretely known is not an easy problem. 
Nevertheless, having them may help to understand the geometry of the moduli spaces of  curves.  In this paper we study an interesting family of non-hyperelliptic Riemann surfaces,  called generalized Fermat curves, where these different descriptions have been studied  and a good understanding of them have been achieved; see for instance \cite{CHQ15, FGHL13,GHL09,HKLP16}.

A closed Riemann surface $S$ is a called a generalized Fermat curve of the type $(k,n)$, where $k,n \geq 2$ are integers, if it admits a group $H \cong {\mathbb Z}_{k}^{n}$ of conformal automorphisms so that the quotient orbifold $S/H$ has genus zero and exactly $(n+1)$ conical points, each one necessarily of order $k$. In this case, the group $H$ (respectively, the pair $(S,H)$) is called a generalized Fermat group (respectively, a generalized Fermat pair) of type $(k,n)$. As a consequence of the Riemann-Hurwitz formula, it can be seen that $S$ has genus
$$g_{k,n}:=\frac{k^{n-1}((n-1)(k-1)-2)+2}{2}.$$

In \cite{FGHL13} it was proved that, for $n=3$ and $k \geq 3$,  a generalized Fermat curve of the type $(k,n)$ has a unique generalized Fermat group of type $(k,n)$ and later, in \cite{HKLP16}, as long $(k-1)(n-1)>2$ (equivalently, $g_{k,n}>1$), this uniqueness property was proved to be true in general. 
 This fact asserts that 
the moduli space of generalized Fermat curves of the type $(k,n)$ can be identified with 
the moduli space of  orbifolds of genus zero with $(n+1)$ cone points, each one of order $k$. In particular, 
generalized Fermat curves of type $(k,n)$ provide a $(n-2)$ complex dimensional family inside
the moduli space of surfaces of genus $g_{k,n}$ (those of type  $(k,2)$ are exactly the classic 
Fermat curves of degree $k$). Also, this facilitates the computation of  the extra automorphisms of
a generalized Fermat curve (see the proof of Corollary 9 of \cite{GHL09}).

Let $(S,H)$ be a generalized Fermat pair of type $(k,n)$, where $(k-1)(n-1)>2$. A representation of $S$ as an algebraic curve and an uniformizing Fuchsian group was provided in \cite{GHL09} and, for $k$ a prime integer, an isogenous decomposition of its Jacobian variety was obtained in \cite{CHQ15}. In fact, 
if the orbifold $S/H$ is uniformized by the Fuchsian group $\Gamma \cong\langle x_1,...,x_{n+1}: x_1^k=\cdots x_{n+1}^k=x_1\cdots x_{n+1}=1\rangle$, then 
its derived subgroup $\Gamma'$ is torsion free, it uniformizes $S$ and 
$H$ corresponds to $\Gamma/\Gamma'$. This, in particular, asserts that $S$ is a maximal Abelian branched covering space of the orbifold $S/H$. 
Let $\rho:S \to \widehat{\mathbb C}$ be a  regular branched cover whose deck group is $H$. Up to post-composition 
by a suitable M\"obius transformation, we may assume the branch values of $\rho$ to be given by 
$\infty$, $0$, $1$, $\lambda_{1},\ldots, \lambda_{n-2}$. Then, $(S,H)$ is isomorphic to the pair
$(\C{k}{n}{\lambda},H)$  (by abuse of notation we use $H$ in both contexts), where 
\begin{equation}\label{star}
 \C{k}{n-2}{\lambda}:=\left \{ \begin{array}{rcc}
              x_0^k+x_1^k+x_2^k&=&0\\
              \lambda_1x_0^k+x_1^k+x_3^k&=&0\\
              \vdots \hspace{1cm}&\vdots &\vdots\\
              \lambda_{n-2}x_0^k+x_1^k+x_n^k&=&0\\
             \end{array}\right  \}\subset {\mathbb{P}}^n,  
\end{equation}
and $H$ is generated by the restrictions of the linear transformations
\begin{center}
 $\varphi_j([x_0:\cdots:x_j:\cdots:x_n]):=[x_0:\cdots:w_kx_j:\cdots:x_n]$, where $w_k:={\mathrm{e}}^{\frac{2\pi i}{k}}$.
\end{center}

The set of fixed points of  $\varphi_j$ in $\C{k}{n-2}{\lambda}$ is ${\mathrm{Fix}}(\varphi_j):=F_j \cap \C{k}{n-2}{\lambda}$,
where $F_j$ is the hyperplane $\{x_j:=0\}\subset {\mathbb{P}}^n$. 
Set $F:=\cup_{j=0}^n{\mathrm{Fix}}(\varphi_j)={\mathrm{Fix}}\;H$.
In this algebraic model, $\rho([x_0,...,x_n])=-(x_{1}/x_{0})^{k}$. The above produces an analytic embedding 
$S \hookrightarrow  \C{k}{n-2}{\lambda}  \subset {\mathbb{P}}^n$, called the standard embedding of the 
generalized Fermat curve $S$.
In \cite{HKLP16} it  was observed that the set hyperosculating points 
of such standard embeddings is $F$; in particular, these are Weierstrass points of $S$. The Weierstrass points are important in the geometry of Riemann surfaces,  and in general the determination of all these  points  together their respective weights remains a difficult problem, including for classically known curves. In the case of the classic Fermat curves (that is, $n=2$), in 1950 Hasse \cite{Has50} computed the Weierstrass weight of the 
hyperosculating points.    Leopoldt observed that for $k\geq 5$ the points $[1:\alpha: \sqrt[k]{2}\beta],
[1:\sqrt[k]{2}\beta:\alpha]$ and  $ [\sqrt[k]{2}:\beta: \alpha\beta]$, were $\alpha$ (resp. $\beta$) is a $k$-th root of $1$ (resp. $-1$) are $3k^2$ 
new Weierstrass points of the Fermat curve (for more information see Rohrlich's article \cite{Roh82}).  
 In 1999 Watanabe \cite{Wat99} showed that in the case $k=6$ additional Weierstrass points exist.
The Weierstrass weight of points fixed by involutions in the case $k \in \{9,10\}$ was obtained by 
Towse in \cite{Tow00}.

In this work we study the  Weierstrass weight of the hyperosculting points of the standard embedding of generalized Fermat curves of type $(k,n)$ when $(k-1)(n-1)>2$. We provide an  optimal lower bound 
(this being sharp in a dense open set  of  the moduli space of the generalized Fermat curves) for the Weierstrass weight of these points.

\section{Preliminaries}\label{sec:hosc}

\subsection{Moduli of generalized Fermat curves}\label{sec:pre-mod-esp}
Let $\mathcal{F}(k,n)$ be the locus, in the moduli space ${\mathcal M}_{g_{k,n}}$ of curves of genus $g_{k,n}$, formed by all the (classes) of generalized Fermat curves of  type $(k,n)$.  The space $\mathcal{F}(k,n)$ is isomorphic to the moduli space $\mathcal{M}_{0,n+1}$ of the unordered $(n+1)$ punctured sphere (see section $4.2$ of \cite{GHL09}).  
Let us consider the affine variety (in fact, a domain in ${\mathbb C}^{n-2}$)
$$\mathcal{P}_n:=\{(\lambda_1,...,\lambda_{n-2})\in
\mbox{(${\mathbb{C}}-\{0,1\}$)}^{n-2}\mid \lambda_i\neq \lambda_j\} \subset {\mathbb C}^{n-2}.$$

To each $(\lambda_1,...,\lambda_{n-2})\in \mathcal{P}_n$ we associate the $(n+1)$-tuple $(\gamma_1,\gamma_2,\gamma_3,\gamma_4,...,\gamma_{n+1})=(\infty, 0,1,\lambda_1,...,\lambda_{n-2})$. Given an element $\sigma \in \mathfrak{S}_{n+1}$, 
 the permutation group of $n+1$ elements, we can form the $(n+1)$-tuple  
  $(\gamma_{\sigma^{-1}(1)},\gamma_{\sigma^{-1}(2)},...,\gamma_{\sigma^{-1}(n+1)})$. 
Let  $T_{\sigma}\in {\rm PSL}(2,\mathbb{C})$  be the unique M\"obius Transformation satisfying $T_{\sigma}(\gamma_{\sigma^{-1}(1)})=\infty$, 
$T_{\sigma}(\gamma_{\sigma^{-1}(2)})=0$, and $T_{\sigma}(\gamma_{\sigma^{-1}(3)})=1$, and set
$$\sigma \cdot (\lambda_{1},\ldots, \lambda_{n-2})=(T_{\sigma} (\gamma_{\sigma^{-1}(4)}),...,T_{\sigma} (\gamma_{\sigma^{-1}(n+1)}) \in {\mathcal P}_{n}.$$

The above provides an action of ${\mathfrak S}_{n+1}$ as a group of holomorphic automorphisms on ${\mathcal P}_{n}$
$$
\mathfrak{S}_{n+1}\times \mathcal{P}_n\rightarrow  \mathcal{P}_n; (\lambda_{1},\ldots, \lambda_{n-2})\mapsto 
\sigma \cdot (\lambda_{1},\ldots, \lambda_{n-2}).
$$

The above action is faithful for $n \geq 4$. For $n=3$ there is a subgroup isomorphic to ${\mathbb Z}_{2}^{2}$ acting trivially on ${\mathcal P}_{3}$.
In \cite{GHL09} it was observed that two generalized Fermat curves $\C{k}{n-2}{\lambda}$ and $\C{k}{n-2}{\mu}$ are isomorphic if and only if $(\lambda_{1},\ldots,\lambda_{n-2})$ and $(\mu_{1},\ldots,\mu_{n-2})$ are in the same ${\mathfrak S}_{n+1}$-orbit. This, in particular, 
permits us to realize the moduli space $\mathcal{F}(k,n)$ as a geometric quotient 
$\mathcal{P}_n/\mathfrak{S}_{n+1}$, that is, as a  complex (affine) variety and that the canonical projection
map $\Pi:  \mathcal{P}_n\to  \mathcal{P}_n /{\mathfrak S}_{n+1}$ is an open morphism.

\subsection{Hyperosculating points}
Next, we will briefly review the general theory of the Pl\"ucker 
formulas in the case of smooth curves. The purpose of this is not to review extensively the theory,
but to present a self contained overview of the parts of the theory relevant to us. 
All the results presented in this section can be found in \cite{GrHa94}.
Let $C \subset {\mathbb{P}}^n$ be a projective smooth curve. For a $l$-plane
$P\subset {\mathbb{P}}^{n}$, $1\leq l\leq n-1$, the  multiplicity of $P$ in $p$ is
\begin{center}
$\mult{P}{C}:=\mbox{Order of contact of  $P$ and $C$ in $p$}$. 
\end{center}

It is known that, for each $p \in C$, there exists a unique $l$-plane, called the {\it osculating $l$-plane} and denoted by $\osc{l}{p}$, such that 
$\mult{\osc{l}{p}}{C}\geq l+1$, and that there exists at most a finite number of points $p\in C$ such that
 $\mult{\osc{l}{p}}{C}> l+1$. We say that $p\in C$ is  a hyperosculating point if 
$$\mult{\osc{n-1}{p}}{C}>n.$$

\begin{remark}
If $C$ is a non hyperelliptic curve of genus $g \geq 3$, then the hyperosculating points of the canonical embedding 
of $C \hookrightarrow {\mathbb P}^{g-1}$ are exactly its  Weierstrass points.
\end{remark}

As the $l$-planes of ${\mathbb{P}}^n$ are in bijective correspondence with the 
dimension $(l+1)$ vector subspaces of ${\mathbb{C}}^{n+1}$, we can define the functions
$$f_l:C\rightarrow \G{l+1}{n+1}; p\mapsto \osc{l}{p},$$
where $\G{l+1}{n+1}$ is the corresponding Grasmannian manifold.

Let $f_0:C\rightarrow {\mathbb{P}}^n$  be the natural embedding defined by the inclusion  $C\subset {\mathbb{P}}^n$,
and let us consider a  local chart $z:U\subset C \rightarrow W \subset {\mathbb{C}} $, $z(p)=0$, 
around the point $p\in C$. Then there exists a neighborhood $W'\subset W$ of $0$,
and a holomorphic vector function 
$$v:W'\rightarrow {\mathbb{C}}^{n+1}\backslash\{0\}:\;z 
\mapsto v(z):=(v_0(z),v_1(z),...,v_n(z)),$$
such that  
$$
f_0(z)=[v_0(z):v_1(z):\cdots:v_n(z)],  \quad \mbox{for all  $z\in W'$.}
$$

Let us consider the holomorphic vector function
$$w:W'\rightarrow \wedge^{l+1}{\mathbb{C}}^{n+1}:\; z\mapsto 
w(z):=v(z)\wedge v'(z)\wedge \cdots \wedge v^{(l)}(z).$$

There exists an integer  $m\geq 0$ such that $w(z)/z^{m}$ 
is a holomorphic vector function which does not vanish in a neighborhood $W''$ of $z=0$. 
By abuse of notation, we may say that $[w(z)]\in {\mathbb{P}}(\wedge^{l+1}{\mathbb{C}}^{n+1})$ for all  $z\in W''$.

Using the Pl\"ucker coordinates, it is possible to see $\G{l+1}{n+1}$ as a subvariety of
${\mathbb{P}}(\wedge^{l+1}{\mathbb{C}}^{n+1})$ and that
$$
f_l(z)=[v(z)\wedge v^{'}(z)\wedge \cdots \wedge v^{(l)}(z) ], \quad \mbox{ for all $z\in W''$}.
$$

In particular, the  maps $f_l$ are holomorphic and independent of the parametrization $v(z)$ chosen.
The curves  $C_l:=f_l(C)$, $0\leq l\leq n-1$  are called the  associated curves of $C$. 
Let us define the following integers:

\begin{itemize}
\item  $b_l(p)$;  the ramification index of $f_l:C\rightarrow C_l$ at the point  $p\in C$.
\item  $b_l=\sum_{p\in C}b_l(p)$; the total ramification index of $f_l:C\rightarrow C_l$.
\item  $d_l$; the number  of osculating $l$-planes of $C$ which intersects a generic $(n-l-1)$-plane of ${\mathbb{P}}^{n}$. Observe that $d_0$ is simply the degree of the curve.
\end{itemize}
   
The following proposition establishes a relationship between the hyperosculating points of $C$ and the ramification indexes of the maps $f_l:C\rightarrow C_l$.
 
\begin{proposition}
The point $p\in C\subset {\mathbb{P}}^n$ is a hyperosculating point if and only if $\sum_{l=1}^{n-1}b_l(p)\geq 1$. 
\end{proposition}
 
In the case of generalized Fermat curves, in \cite{HKLP16} (See Theorem \ref{th:hyperosc}) the ramification indexes were explicitly computed, which allows us to determine the hyperosculating points.  The following theorem will be useful for this purpose.

\begin{theorem}[\bf Pl\"ucker Formulas]
\label{th:plucker}
If $C\subset {\mathbb{P}}^n$ is  a curve of genus  $g$, then 
\begin{center}
$d_{l+1}-2d_l+d_{l-1}=2g-2-b_l$, for all $1\leq l\leq n-1$,
\end{center}
where $d_{-1}=d_n=0$.
\end{theorem}

\subsection{A computational method}
We proceed to describe a method for computing $b_{l}(p)$.
Keeping the notations as above, let $z$ be a local chart around the point $p\in C$ and, in local charts,
$$f_0(z)=[v_0(z):v_1(z):\cdots:v_n(z)].$$

Making linear changes of coordinates, it is possible to prove that there exists $\varphi\in {\mathrm{Aut}}({\mathbb{P}}^n)\cong \PGL{n+1}{{\mathbb{C}}}$ 
such that
$$
\varphi(f_0(z))=[1: z^{1+\alpha_1}+\cdots: z^{2+\alpha_1+\alpha_2}+\cdots:\cdots :z^{n+\alpha_1+\cdots+\alpha_n}+\cdots].
$$

This is called the  normal  form of $f_0$ in $p$. By abuse of notation, we will identify  $\varphi(f_0(z))$ with $f_0(z)$. 
It is possible to verify that the integers $\alpha_j$, $1\leq j\leq n$, only depend on $f_0$ and the point $p\in C$, and neither on the chosen local chart $z$, nor 
the vector function $v(z)$,  nor the automorphism $\varphi$.

\begin{proposition}\label{pr:ind-ram} 
In the above, 
$$
b_l(p):=\alpha_{l+1}, \quad 0\leq l\leq n-1.
$$
\end{proposition}

In particular, as $C$ is a smooth curve, $\alpha_1=0$.

\begin{remark}
\label{re:gap-W}
If $C$ is a non-hyperelliptic curve of genus  $g \geq 3$ and $C\hookrightarrow  {\mathbb{P}}^{g-1}$ 
is a canonical embedding, then 
$$
a_i=i+\sum_{j=1}^{i-1}\alpha_j, \quad 1\leq i\leq g,
$$
are the gap values of $p$. In other words,  $a_{1},\ldots,a_{g}$ are the only $g$ integers where  {\bf  there does not exist}
 a meromorphic function of  $C$ with a pole of order $a_i$ at the point $p$ and holomorphic on $C-\{p\}$.
\end{remark}

\subsection{The natural regular branched coverings of the generalized Fermat curves}
\label{ss:natural-branched-covering}

Let us start with the following general fact (which will be used later) regarding the generalized Fermat curves. 

\begin{remark}\label{re:rest}
Let us consider a generalized Fermat curve of type $(k,n)$ with its respective standard embedding 
$\xymatrix{\C{k}{n-2}{\lambda} \ar@^{(->}[r] & {\mathbb{P}}^{n} }$, 
and a rational map (denoted by the dotted line)  $\psi: {\mathbb{P}}^{n} \dashrightarrow {\mathbb{P}}^{m}$.
 If there exists 
 $0 \leq i<j\leq n$  such that the linear projective space $L_{(i,j)}:=\{[x_0:\cdots:x_n]|x_i=x_j=0\}$ contains the
locus of indeterminacy of the rational map, then  (as the intersection $\C{k}{n-2}{\lambda}\cap L_{(i,j)}$ is the empty set)
 the restriction of $\psi$  to $\C{k}{n-2}{\lambda}$, $$\xymatrix{ \C{k}{n-2}{\lambda} \ar@^{(->}[r] \ar@/^1pc/[rr]&{\mathbb{P}}^{n}\ar@{-->}[r]&{\mathbb{P}}^m},$$
is a well defined morphism.  
\end{remark}

Let us consider the rational  map  
$$\pi: {\mathbb{P}}^{n} \dashrightarrow {\mathbb{P}}^{n-1}: [x_0:\cdots:x_n] \mapsto [x_0:\cdots:x_{n-1}].$$

As the locus of indeterminacy of $\pi$ is the point
$[0:\cdots:0:1]\in L_{(0,1)} \subset {\mathbb{P}}^{n}$, then (by Remark \ref{re:rest}) the restriction of $\pi$ to $\C{k}{n-2}{\lambda}$
is a well defined morphism.  Additionally, $\pi(\C{k}{n-2}{\lambda}))=\C{k}{n-3}{\lambda}$ (when $n=3$, 
this image curve is the classic Fermat curve $C^k$).
By abuse of notation we  denote by $\pi: \C{k}{n-2}{\lambda} \to \C{k}{n-3}{\lambda}$ the restriction of
$\pi$ to $\C{k}{n-2}{\lambda}$. Note that the restricted map  $\pi$,  is a regular branched covering
whose deck covering group is the cyclic group generated the automorphism 
$$\varphi_n([x_0:\cdots:x_n]):=[x_0:\cdots :\mathrm{e}^{\frac{2\pi i}{k}}x_n],$$
so its  branch values  are the images of the $k^{n-1}$ fixed points of $\varphi_{n}$.
This map $\pi$ is compatible with the embeddings of the Fermat curves into projective spaces. 
The quotient group $H/\langle \varphi_{n}\rangle \cong {\mathbb Z}_{k}^{n-1}$ is the generalized Fermat
group of type $(k,n-1)$ of  $\C{k}{n-3}{\lambda}$. The $k^{n-1}$ fixed points of each 
$\varphi_{j}$ ($j=0,\ldots,n$) are permuted under the action of $\varphi_{n}$ (fixing none of them), and 
this set is projected under $\pi$ to the set of $k^{n-2}$ fixed points of the quotient class of $\varphi_{j}$. 
This map will be used to obtain an inductive approach to our problem.

\section{Hyperosculating points of generalized Fermat curves}
In this section we restrict our attention to generalized Fermat curves.
Keeping  the notations fixed at the beginning of the Section \ref{sec:pre-mod-esp}, recall that $\C{k}{n-2}{\lambda}$ is a generalized Fermat curve of the type $(k,n)$, and that  $F:={\mathrm{Fix}} H$.   
 
If $p \in F$, then, by using linear substitutions in the system of equations, we may assume that $p\in {\mathrm{Fix}}(\varphi_1)$, that is, 
$$p:=[1:0:\rho_1:\rho_2:\cdots:\rho_{n-1}],$$
where $\rho_i^{k}=-{\lambda}_{i-1}$, $1\leq i\leq n-1$ (with ${\lambda}_{0}=1$).

Let $f_0:\C{k}{n-2}{\lambda} \rightarrow {\mathbb{P}}^n$ be the standard embedding defined by the inclusion $\C{k}{n-2}{\lambda} \subset {\mathbb{P}}^n$.

The next theorem describes the hyperosculating points of $\C{k}{n-2}{\lambda}$ and the ramification indexes.

\begin{theorem}[\cite{HKLP16}]\label{th:hyperosc}
\mbox{}
\begin{enumerate}
\item The set of hyperosculating points of $\C{k}{n-2}{\lambda}$ is $F$. 
\item  If $p\in F$, then $b_1(p)=k-2$ and $b_l(p)=k-1$,  $l \in \{2,\ldots, n-1\}$.
\end{enumerate}
\end{theorem}
 
A consequence of the above is the following.

\begin{corollary}\label{co:hyperosc}
Let $z$ be  a local chart of  $\C{k}{n-2}{\lambda}$  around the point  $p\in \C{k}{n-2}{\lambda}$. Then the normal form of 
$f_0$ in  $z(p):=0$  is given as follows.
\begin{enumerate}
\item If $p\in F$, then 
$$f_0(z)=[1:z: g_0(z^{k}):g_1(z^{k}):\cdots :g_{i}(z^{k}):\cdots : g_{n-1}(z^{k})],$$
where the  $g_i$ are holomorphic functions such that $g_i(z)=z^{i+1}+\cdots$.

\item If $p\not \in F$, then  
$$f_0(z)=[1:z: z^2+\cdots :\cdots : z^{(n-1)}+\cdots].$$
\end{enumerate}
\end{corollary}

\subsection{The Weierstrass  weight of the hyperosculating points of generalized Fermat curves}
\label{ssec:aut-hiposc}
Let us keep the notations from the previous sections.  
Given a curve $C$, the Weierstrass  weight  of $p\in C$ is 
$$w(p):=\sum_{i=1}^{g}(a_i-i),$$
where the $a_i$ are the Weierstrass gaps of $p$ (see  Remark \ref{re:gap-W}).

In general, the computation of $w(p)$, when $p$ is a Weierstrass point (which is to say $w(p)>0$), 
is not an easy problem. Theorem  \ref{th:hyperosc} asserts that the hyperosculating points of the 
generalized Fermat curve  $\C{k}{n-2}{\lambda}$ are exactly the points in the set $F$.
We will determine an optimal lower bound for the weight of these points and observe that 
this bound is sharp in a  dense open set of the moduli space of the generalized Fermat curves of type $(k,n)$.\\

First, we fix some notations.
Let $I(\C{k}{n-2}{\lambda}):=\langle x_0^k+x_1^k+x_2^k,...,\lambda_{n-2}x_0^k+x_1^k+x_{n}^k\rangle$
be the homogeneous prime ideal 
of  $\C{k}{n-2}{\lambda}$ in ${\mathbb{C}}[x_0,...,x_n]$, 
let $\Gamma(\C{k}{n-2}{\lambda}):={\mathbb{C}}[x_0,...,x_n]/I(\C{k}{n-2}{\lambda})$ 
be the homogeneous coordinate ring of $\C{k}{n-2}{\lambda}$, and let $\mathcal{O}_{{\mathbb{P}}^n}(m)$, $m\in {\mathbb{Z}}$, be the twisting sheaf; for $m\geq 0$ the sheaf  $\mathcal{O}_{{\mathbb{P}}^n}(m)$ is generated by the forms  of degree $m$ of ${\mathbb{C}}[x_0,...,x_n]$.

Let us consider the  sheaf over $\C{k}{n-1}{\lambda}$ given by
$$\mathcal{O}_{\C{k}{n-2}{\lambda}}(m):=f_0^{\star}\mathcal{O}_{{\mathbb{P}}^n}(m),$$
where $f_0:\C{k}{n-2}{\lambda}\hookrightarrow {\mathbb{P}}^n$ is the natural embedding of the generalized Fermat curves.
To simplify the notation, when it is clear that we are referring to the sheaf $\mathcal{O}_{\C{k}{n-2}{\lambda}}(m)$  we will simply use the notation
$\mathcal{O}(m)$.  Observe that $H^{0}(\C{k}{n-2}{\lambda},\mathcal{O}(m))=\Gamma(\C{k}{n-2}{\lambda})_m$,
where  $\Gamma(\C{k}{n-2}{\lambda})_m$ are the forms of degree  $m$ of $\Gamma(\C{k}{n-2}{\lambda})$.

As  $I:=I(\C{k}{n-2}{\lambda})$ is a homogeneous prime ideal, we have that 
 $I=\bigoplus_{m\geq 0}I_m$ and $\Gamma(\C{k}{n-2}{\lambda})= \bigoplus_{m\geq 0}\mathbb{C}[x_0,...,x_n]_m/I_m$,
where $I_m$ are the forms of degree $m$ of $I$ (observe that $I_0=\cdots =I_{k-1}=0$). 
In particular we have a surjective linear transformation 
$${\mathbb{C}}[x_0,...,x_n]_m\rightarrow \Gamma(\C{k}{n-2}{\lambda})_m=
 \mathbb{C}[x_0,...,x_n]_m/I_m  \hspace{0.5cm} (\star) $$

If ${\mathbb{P}}(\Gamma(\C{k}{n-2}{\lambda})_m)$ denotes the projective space associated to the vector
space $\Gamma(\C{k}{n-2}{\lambda})_m$, then the  linear transformation $(\star)$  induces an embedding
$${\mathbb{P}}(\Gamma(\C{k}{n-2}{\lambda})_m)\subset {\mathbb{P}}({\mathbb{C}}[x_0,...,x_n]_m)\cong {\mathbb{P}}^{d(m)},$$ 
where $d(m)=\binom{n+m}{m}-1$. In fact we can realize $\mathbb{P}(\Gamma(\C{k}{n-2}{\lambda})_m)$
as a linear projective space of ${\mathbb{P}}^{d(m)}$.

If $C$ is an algebraic curve, then we denote by $\omega_C$ its respective canonical sheaf. 
When given the context it is clear that we are referring to a curve $C$, we use the notation $\omega$ in place of $\omega_C$. 

As $\omega\cong \mathcal{O}_{\C{k}{n-2}{\lambda}}(r)$, $r:=(n-1)(k-1)-2$, (\cite[page 188]{Har77}), the Veronesse map of degree $r$,
$\nu_r:{\mathbb{P}}^n\rightarrow {\mathbb{P}}^{d(r)}$, permits us to obtain the canonical embedding  $f_c$. 
  
\begin{center}
$\xymatrix{\C{k}{n-2}{\lambda} \ar@^{(->}[rr]_{f_0} \ar@/^1pc/[rrr]^{f_c:=\nu_r\circ f_0}&
& {\mathbb{P}}^{n}\ar@{-->}[r] \ar@/_1pc/[rr]_{\nu_r} &{\mathbb{P}}(\Gamma(\C{k}{n-2}{\lambda})_r) \ar@^{(->}[r]& {\mathbb{P}}^{d(r)}. }$
\end{center}
 
\begin{remark}
To define the rational maps 
$\xymatrix{{\mathbb{P}}^{n}\ar@{-->}[r]  &{\mathbb{P}}(\Gamma(\C{k}{n-2}{\lambda})_r)}$
it is necessary to fix a basis $\mathcal{B}$ of the vector space $\Gamma(\C{k}{n-2}{\lambda})_r)$.  Observe that for $(k,n)\neq(2,5)$ the rational map is a well defined morphism. 
As for $(k,n)\neq(2,5)$ we have that $r\neq k$, there exists a basis $\mathcal{B}$ of  $\Gamma(\C{k}{n-2}{\lambda})_r)$ such that $x_0^r,...,x_n^r\in \mathcal{B}$.  In the case $(k,n)=(2,5)$ we can suppose that $x_0^2, x_1^2 \in \mathcal{B}$.
\end{remark}

The normal form of $f_0$ in $p\in F$ (Corollary \ref{co:hyperosc}), will provide information about the normal form
 of canonical embedding $f_c$ at $p\in F$. More precisely, as 
 all elements of $p(x_0,x_1,...,x_n)\in \Gamma(\C{k}{n-2}{\lambda})_m$ 
 can be written uniquely in the form
 $$p(x_0,x_1,...,x_n)=\sum_{j=1}^{k-1}x_n^{j}q_j(x_0,...,x_{n-1}),$$
 where $q_j(x_0,...,x_{n-1})\in  \Gamma(\C{k}{n-2}{\lambda})_{m-j}$,
 it follows that the vector space 
 $\Gamma(\C{k}{n-2}{\lambda} )_m$ 
  has the following decomposition
 $$\displaystyle \Gamma(\C{k}{n-2}{\lambda})_m:=\bigoplus_{j=0}^{k-1} x_n^{j}Q(m-j),$$
 where $Q(m-j)\subset \Gamma(\C{k}{n-2}{\lambda})_{m-j}$.\\

Let us choose a basis $v_0(x_0,...,x_{n-1}),\cdots,v_{t_j}(x_0,...,x_{n-1})$  of the vector space  $Q(r-j)$, where $t_j:=\dim_{\mathbb{C}}Q(r-j)-1$, and $0\leq j\leq r-1$. Therefore we can construct a rational map:
$$\xymatrix{{\mathbb{P}}^{n}\ar@{-->}[r] &{\mathbb{P}}(Q(r-j))}:[x_0:x_1:\cdots x_n]\mapsto [v_0:v_1:\cdots v_{t_j}].$$
	
We  may assume that $v_0(x_0,...,x_{n-1}):=x_0^{r-j}$, and $v_1(x_0,...,x_{n-1}):=x_1^{r-j}$. 
 In fact, if $x_0^{r-j}, x_1^{r-j}$  are linearly dependent in $Q(r-j)$, 
then there exists $a\in \mathbb{C}$ such that $x_0^{r-j}+ax_1^{r-j}\in I(\C{k}{n-2}{\lambda})$. 
 As  $I(\C{k}{n-2}{\lambda})$ is a prime ideal,  there exists  $a'\in \mathbb{C}$ such that $x_0-a'x_1\in I(\C{k}{n-2}{\lambda})$,
  a contradiction since $I_1=0$. 
 In thus way, we observe that  the locus of indeterminacy of the  above rational map  is contained in the linear space $L_{(0,1)}:=\{[x_0:\cdots:x_n]|x_i=x_1=0\}$. 
   By Remark  \ref{re:rest}, its restriction to $\C{k}{n-2}{\lambda}$ 
is a well defined morphism, which we denote by the symbol $k_j$.
In this manner we have constructed the following diagram:	

$$\xymatrix{ \C{k}{n-2}{\lambda} \ar@^{(->}[r] \ar@/^1pc/[rr]^{k_j}&{\mathbb{P}}^{n}\ar@{-->}[r]&{\mathbb{P}}(Q(r-j))},$$
 
The  morphisms  $k_j$ allow us to study the morphism $f_c$ at $p\in F$.
 
Now, as seen in Section \ref{ss:natural-branched-covering}, if we consider restriction of the rational map  
$$\pi: {\mathbb{P}}^{n} \dashrightarrow {\mathbb{P}}^{n-1}: [x_0:\cdots:x_n] \mapsto [x_0:\cdots:x_{n-1}]$$
to the generalized Fermat curve $\C{k}{n-2}{\lambda}$, then we obtain the morphism  (for $n=3 $, 
the curve $\C{k}{n-3}{\lambda}$ is the classic Fermat curve $C^k$)
$$\pi:\C{k}{n-2}{\lambda}\rightarrow \C{k}{n-3}{\lambda};[x_0:\cdots x_n]\mapsto [x_0:\cdots:x_{n-1}].$$

Observe that $\pi$ is a Galois branched covering of degree $k$ defined by quotienting 
by the cyclic group of order $k$ generated by the automorphism 
$$\varphi_n([x_0:\cdots:x_n]):=[x_0:\cdots:x_{n-1}:w_kx_n], \quad w_k:={\mathrm{e}}^{\frac{2\pi i}{k}}.$$

In particular, the morphism $\pi$ defined in the local chart $z$, around the point $p$, is of the form  
$$\zeta:=\pi(z)=z^{k}.$$

By the construction, we may see that $\pi$ factorizes the morphism $k_j$, and we obtain the following diagram

\begin{equation}\label{diagrama2}
\xymatrix{ \C{k}{n-2}{\lambda} \ar[d]^{\pi}  \ar@{^{(}->}[rr] \ar@/^1pc/[rrrr]^{k_j}  & & 
{\mathbb{P}}^{n}\ar@{-->}[d]^{\pi} 
\ar@{-->}[rr]&& {\mathbb{P}}(Q(r-j)) \\
\C{k}{n-3}{\lambda} \ar@{^{(}->}[rr]\ar@/_3pc/[rrrru]_{g_j} & &{\mathbb{P}}^{n-1}\ar@{-->}[rru]&&}
\end{equation}

Analogous to what we have previously seen, the locus of indeterminacy (this can be empty set) of the rational map
  $\xymatrix{{\mathbb{P}}^{n-1}\ar@{-->}[r] &{\mathbb{P}}(Q(r-j))}$
 is contained in $L_{(0,1)}:=\{[x_0:\cdots:x_{n-1}
|x_0=x_1=0\}$,  so $g_j$ is a well defined morphism.
 To study the morphism $k_j$ we strongly use the morphism $g_j$.
 
The following result gives us the dimension of the $Q(r-j)$.  

\begin{proposition} \label{pr:dim-s} 
If $s(r-j):= 
 \dim_{{\mathbb{C}}}Q(r-j)$, where  $r:=(n-1)(k-1)-2$ and  $0\leq j\leq k-1$, then
$$s(r-j)=\frac{1}{2} k^{n-2} ((n(k-1) -2 - 2 j)+ \delta_{k-1, j},$$
 where $\delta_{k-1, j}$ is the Kronecker delta.
 \end{proposition}
\begin{proof} Let us fix a generalized Fermat curve $\C{k}{n-2}{\lambda}$, and consider the generalized Fermat curve $S'=\C{k}{n-3}{\lambda}$ of type $(n-1,k)$. We note that 
$$H^{0}(S',\mathcal{O}_{S'}(r-j))\cong Q(r-j).$$

For $m\in{\mathbb{Z}}$,  let $h'(m)$ denote the dimension of the global section space $H^{0}(S',\mathcal{O}_{S'}(m))$ over ${\mathbb{C}}$.
Recall that $\omega_{S'}\cong \mathcal{O}_{S'}((n-2)(k-1)-2)$. By the Riemann-Roch Formula, 
$$h'(-k+1+j))-h'(r'-(-k+1+j))=(-k+1+j)k^{n-2}-\left(\frac{k^{n-2} r' +2}{2}\right)+1,$$
where $r':=(n-2)(k-1)-2$. Since $s(r-j)= h'(r'-(-k+1+j))$ and  $h'(-k+1+j))=\delta_{k-1, j}$, we obtain the desired equality.
\end{proof}

Next, we estimate, from below, the weight of the points in $F$.

\begin{theorem}\label{th:w-peso}
Let $p \in F$ and let $w(p)$ be its weight. If $n\geq 3$, then
$$\widehat{w}(p):=\frac{1}{24}(k-1) (k^{n-1}-2) (k^{n} + k^{n-1}-12)\leq w(p).$$
Moreover, there exists a dense open set  of $\mathcal{F}(k,n)$ where equality holds.  
\end{theorem}

\begin{remark}\label{ej:cla-fer}
In the case of a generalized Fermat curve of the type  $(k,2)$, $k\geq 4$, (a classic Fermat curve), 
it is known that (see \cite{Has50,Roh82,Wat99}) the weight of a point $p\in F$ is
$w(p)=\frac{1}{24}(k-1)(k-2)(k-3)(k+4)$, which shows that equality in Theorem \ref{th:w-peso} holds for the classic case.
\end{remark}
 
Before proving Theorem \ref{th:w-peso}, we discuss two examples. In the first
one we observe that the equality, in the previous theorem, holds for a generalized 
Fermat curve of type $(2,4)$ and in the second one we provide examples for which the bound is not sharp.
 
\begin{example}
\label{ej:(2,4)}
Let  $C$ be a generalized Fermat curve of the type $(k,n)=(2,4)$, and  $p$ be in $F$. As $H^{0}(C,\omega_C)\cong H^{0}(C,\mathcal{O}_C(1))={\mathbb{C}}[x_0,x_1,x_2,x_3,x_4]_1$  (forms of degree $1$ of the polynomial ring  ${\mathbb{C}}[x_0,x_1,x_2,x_3]$),
It follows that the map $\xymatrix{ C \ar[r]^{f_c}&{\mathbb{P}}^{4}}$ is the canonical embedding.
Using the normal form of $f_{c}$ in $p$, we obtain that the gap values of $p$ are
 $a_1=1$, $a_2=2$, $a_3=3$, $a_4=5$, $a_5=7$. In this way,
 $\hat{w}(p)=w(p)=3$. Observe that, by virtue of Theorem \ref{th:hyperosc}, in this case the Weierstrass points are exactly the points of hyperosculation $F$.  
 \end{example}

\begin{example}\label{ej:bound}
Consider the following generalized Fermat curve of the type $(5,3)$  
\begin{center}
$C^5(-1):=\left \{ \begin{array}{rcc}
x_0^5+x_1^5+x_2^5&=&0\\
-x_0^5+x_1^5+x_3^5&=&0
\end{array}\right .$
\end{center}
and $p=[1:1:-\sqrt[5]{2}:0]$. 
In this case, $\hat{w}(p)<w(p)$.  In Section \ref{Sec:finalejemplo} we provide a proof of this fact.  
\end{example}

\subsection{Proof of Theorem  \ref{th:w-peso}}
Because of Example \ref{ej:cla-fer}, we only need to consider $n\geq 3$.  
 Let us consider the generalized Fermat curve $\C{k}{n-2}{\lambda}$, and $p\in F$.  Without loss of generality, 
 we can suppose that 
$p:=[1:\rho_1:\rho_2:\cdots :\rho_{n-1}:0]$, where $\rho_i^k=-\lambda_{n-2}-\lambda_{i-2}$,  $\lambda_{-1}=0$, and  $\lambda_{0}=1$
(it suffices to use the linear substitutions in the system of equations $\C{k}{n-2}{\lambda}$).
Also, let us recall the commutative diagram \eqref{diagrama2}.

The set $D(\lambda_1,...,\lambda_{n-2})=\pi({\mathrm{Fix}}(\varphi_n))\subset \C{k}{n-3}{\lambda}$, is the set of branch values of the regular branched covering map $\pi:\C{k}{n-2}{\lambda} \to \C{k}{n-3}{\lambda}$ whose deck group is $\langle \phi_{n} \rangle \cong {\mathbb Z}_{k}$. Define the sets
$$\mho_j:=\{(\lambda_1,...\lambda_{n-2})\in\mathcal{M}_{0, n+1} \mid D(\lambda_1,...,\lambda_{n-2})
 \cap \HC{j}{k}{n-3}{\lambda}=\emptyset \},$$
$$\mho:=\cap_{j=0}^{k-1}\mho_{j},$$
where $\HC{j}{k}{n-3}{\lambda}$ is the set of hyperosculating points of the map 
$$g_j:\C{k}{n-3}{\lambda}\hookrightarrow {\mathbb{P}}(Q(r-j)).$$

\begin{lemma}\label{le:open-zariski}
For each $j \in \{0,\ldots,k-1\}$, the set  $\mho_j$ is a  dense open set of  $\mathcal{M}_{0,n}$; in particular, $\mho$ is also a non-empty dense open set.
\end{lemma}
\begin{proof}
Let us consider the set 
$$\mho'_j:=\{(\lambda_1,...\lambda_{n-2})\in\mathcal{P}_n \mid D(\lambda_1,...,
\lambda_{n-2})\cap \HC{j}{k}{n-3}{\lambda}=\emptyset \}.$$  

As  $\Pi: \mho'_j\rightarrow \mho_j$ is an open  surjective map 
(see section  \ref{sec:pre-mod-esp}), it suffices to prove that $\mho'_j$ is an open set in the domain ${\mathcal P}_{n} \subset {\mathbb C}^{n-2}$.

In the following, when we use the notation $\hat{\lambda}_i$ we suppose that the value of $\lambda_i$ is fixed.

Let us first verify that $\mho'_j$ is non-empty.  Fix a point $(\hat{\lambda}_1,...,\hat{\lambda}_{n-3}) 
\in {\mathcal P}_{n-1}$ and consider the slice in ${\mathcal P}_{n}$ given by the points of the form
$(\hat{\lambda}_1,...,\hat{\lambda}_{n-3}, \lambda_{n-2}) \in {\mathcal P}_{n}$. We proceed to prove
that in such a slice only finitely many points cannot belong to $\mho_{j}$. For this fact, 
we only need to observe that the set
$$\{q:=\pi(p)\in \C{k}{n-3}{\hat{\lambda}}\mid \lambda_{n-2}\in {\mathbb{C}}-\{0,1, \hat{\lambda}_1,...,\hat{\lambda}_{n-3} \},\; p\in {\mathrm{Fix}}{(\varphi_n)}\}$$
 has infinitely many points and $\HC{j}{k}{n-3}{\hat{\lambda}}$ is a finite set.

Now, we proceed to check that $\mho'_{j}$ is open.  
Let us fix $(\hat{\lambda}_0,...,\hat{\lambda}_{n-2}) \in \mho'_{j}$. 
Let $q'\in D(\hat{\lambda}_0,...,\hat{\lambda}_{n-3})$ be such that $q'\not\in \HC{j}{k}{n-3}{\hat{\lambda}}$ and 
$p'\in {\mathrm{Fix}}{(\varphi_n)}$ such that $\pi(p')=q'$. 

For each $(\lambda_{1},\ldots,\lambda_{n-2}) \in {\mathcal P}_{n}$, let us consider a point 
$p\in {\mathrm{Fix}}{(\varphi_{n})}\subset \C{k}{n-2}{\lambda}$ and set $q=\pi(p)$. 
Observe that there are no technical problems in supposing that $(p,q)=(p',q')$ 
when  $(\lambda_{1},\cdots, \lambda_{n-2})=(\hat{\lambda}_{1},\cdots, \hat{\lambda}_{n-2})$. 
Recall that we may assume $p=[1:\rho_1:\cdots :\rho_{n-1}:0]$, where $\rho_i^k=-\lambda_{n-2}-\lambda_{i-2}$,  $\lambda_{-1}=0$ and $\lambda_{0}=1$.

If $z$ is a local chart around $p$, $z(p)=0$, then there exists a neighborhood $\Omega_p  \subset {\mathbb{C}}$ of $0$,  such that the map $f_0:\Omega_p \rightarrow \C{k}{n-2}{\lambda}\subset {\mathbb{P}}^{n}$, defined naturally by the embedding $\C{k}{n-2}{\lambda}\subset {\mathbb{P}}^{n}$, has the form
$$ \displaystyle f_0(z)=[1:h_1(z^{k}):h_2(z^{k}): \cdots: h_{n-1}(z^{k}):z],$$ 
where $h_i(0)=\rho_{i}$.

By the construction, there exists a local chart $\zeta$ of $\C{k}{n-3}{\lambda}$ around the point $q:=\pi(p)$, and a local parametrization 
$$\tilde{f}_0:\Omega'_q\subset {\mathbb{C}}\rightarrow \C{k}{n-3}{\lambda}\subset {\mathbb{P}}^{n-1}:\zeta\mapsto \tilde{f}_0(\zeta)=[1:h_1(\zeta):h_2(\zeta): \cdots: h_{n-1}(\zeta)].$$

Recall that the morphism $\pi$ is defined in local charts as $\zeta=\pi(z)=z^k$.  Additionally, 
if we consider a Taylor expansion of the $h_i(\zeta)$  we can observe that the coefficients are analytic  functions in the variables 
$\lambda_1,...,\lambda_{n-2}$.

Now let us consider the map $g_{j}:\C{k}{n-3}{\lambda}\rightarrow {\mathbb{P}}(Q(r-j))$. We note that, using the  local parametrization  of $\tilde{f}_0(\zeta)$ around the point $q$, a local parametrization 
 of   $g_{j}$ around the point $q$  can be found. Let  $g_{j}(\zeta):=[1:h'_1(\zeta):\cdots:h'_{t_j}(\zeta)]$  be  this 
 local parametrization of $g_{j}$,  where $t_j:=s(r-j)-1= 
 \dim_{{\mathbb{C}}}Q(r-j)-1$.
By construction the coefficients of $h'_i(\zeta)$ are analytic functions in the variables
$\lambda_1,...,\lambda_{n-2}$.

Given a formal series $l(\zeta)$, let $Tl$ be the vector column formed by all the coefficients of the formal series $l(\zeta)$
until the grade $t_j$.  We consider the following analytic function defined over $\mathcal{P}_n$ 

$$r(\lambda_1,...,\lambda_{n-2})=\det(T1,Th'_1,...,Th'_{t_j}).$$
	
Fixing the values  $\lambda_1,...,\lambda_{n-2}$, the point  $q\in \C{k}{n-3}{\lambda}$  is not a point of $\HC{j}{k}{n-3}{\lambda}$ if and only 
if $r(\lambda_1,...,\lambda_{n-2})$ is not zero.  As $q'$ is not a point of $\HC{j}{k}{n-3}{\hat{\lambda}}$,
we have that $r(\hat{\lambda}_1,...,\hat{\lambda}_{n-2})\neq 0$ (in particular, $r$ is not identically zero). 
The set $\mho'_j\subset \mathcal{P}_n$,  where the analytic function $r$ does not vanish, is the sought after open set. 
\end{proof}

Now, let us recall the map:
$$\xymatrix{ \C{k}{n-3}{\lambda} \ar[r] \ar@/^1pc/[rr]^{g_j}&{\mathbb{P}}^{n-1}\ar@{-->}[r]&{\mathbb{P}}(Q(r-j))}, \quad 0\leq j \leq r-1,$$
where $r=(n-1)(k-1)-2$, $Q(r-j)$  is the ${\mathbb{C}}$-vector space of the proposition \ref{pr:dim-s}.
Considering an automorphism of  $ {\mathbb{P}}(Q(r-j))$,  we can obtain the normal form of $g_j$ in $\pi(p)$:
 $$g_{j}(\zeta)=[1:g_{(1\; j)}(\zeta):g_{(2\;j)}(v):\cdots :g_{(t_j\;j)}(\zeta)],$$  
 where $t_j:=s(r-j)-1= \dim_{{\mathbb{C}}}Q(r-j)-1$,  and  the following inequalities are satisfied for each $0\leq j\leq r-1$:

$$(\star) \left \{\begin{array}{lll}
i\leq l_i:= {\mathrm{Ord}} g_{(i\; j)}(v),& \mbox{for all}& 1\leq i\leq t_j,\\
{\mathrm{Ord}} g_{(i\; j)}(v)< {\mathrm{Ord}} g_{(i+1\; j)}(v),& \mbox{for all} & i\geq 1.
\end{array} \right.$$

\begin{remark} \label{re:li=i}
For fixed $j \in \{0, \ldots, r-1\}$ the equality $l_i=i$, $1\leq i\leq t_j$, is valid when $\pi(p)$ is not a
 hyperosculation point of the  morphism  $\C{k}{n-3}{\lambda}\rightarrow {\mathbb{P}}(Q(r-j))$.  
\end{remark}

The idea of the rest of the proof is to construct the normal form of the canonical
embedding $f_c:\C{k}{n-2}{\lambda}\rightarrow {\mathbb{P}}^{g-1}$ at the point $p$ using the functions $g_{(i\; j)}$. We divide the proof into two cases:  
\begin{enumerate}
\item[] {\bf Case 1:} $(n-1)(k-1)-2<k$.
\item[] {\bf Case 2:} $(n-1)(k-1)-2\geq k$.
\end{enumerate}

\smallskip
 \noindent
 {\bf Case 1:}
 Let us suppose that $(n-1)(k-1)-2<k$; that is $(k,n) \in \{(2,4),(3,3)\}$.
 In example \ref{ej:(2,4)} we have analyzed the case $(k,n)=(2,4)$. 
  Let us consider a generalized Fermat curve $C^3(\lambda)$ of type $(3,3)$ and the morphism $\pi:C^3(\lambda)\rightarrow C^3$ , previously defined.  
 Observe that  $H^{0}(C^3(\lambda):\omega)\cong H^{0}(C^3(\lambda):\mathcal{O}(2))={\mathbb{C}}[x_0,x_1,x_2,x_3]_2$ 
 (forms of degree $2$ of the polynomial ring  ${\mathbb{C}}[x_0,x_1,x_2,x_3]$).
 Then, the canonical embedding $f_c$ is given by the following composition:
 \begin{center}
 $\xymatrix{ C^3(\lambda) \ar[r]_{f}\ar@/^1pc/[rr]^{f_c} &{\mathbb{P}}^{3}\ar[r]_{\nu_{2}}&{\mathbb{P}}^9} $.
 \end{center}
Now let us consider the following morphism:  
\begin{center}
   $\xymatrix{ C^{3} \ar[r] \ar@/^1pc/[rr]^{g
   :=\nu_{2}\circ \tilde{f}_0}&{\mathbb{P}}^{2}\ar[r]&{\mathbb{P}}(Q(2))\cong {\mathbb{P}}^5}$, 
 \end{center}
 where $Q(2)$ is the ${\mathbb{C}}$-vector space of Proposition \ref{pr:dim-s}, 
 $\tilde{f}_0$ the natural embedding of $C^{3}$.  
	Let us consider the normal form of $\tilde{f}_0$, and of $g$:
 \begin{center}
   $\tilde{f}_{0}(\zeta)=[1:\tilde{f}_{(0\;1)}(\zeta):\tilde{f}_{(0\;2)}(\zeta)]$,  $g(\zeta)=[1:g_{1}(\zeta):g_{2}(\zeta):\cdots :g_{5}(\zeta)]$, 
 \end{center}

When $\pi(p)$ is not a hyperosculating point of  $\tilde{f}_0:C^3\rightarrow {\mathbb{P}}^2$, we obtain that  ${\mathrm{Ord}} \tilde{f}_{(0\;i)}(\zeta)=i$, $i\leq 2$. 
 We also obtain that ${\mathrm{Ord}} g_{i}(\zeta)=i$, $i\leq 4$ and  ${\mathrm{Ord}} g_{5}(\zeta)\geq 5$.  
 Observe that ${\mathrm{Ord}} g_{5}(\zeta)=5$ if  and only if $\pi(p)$ is not a hyperosculating point of the embedding $g:C^3\rightarrow {\mathbb{P}}^5$. As in this case, $\zeta:=\pi(z):=z^3$, if we define the functions 
\begin{center}
$h_{i}(z):=\left \{ \begin{array}{ccc} 
g_{i}(z^3), & \mbox{if} & i=1,\\
z\tilde{f}_{(0\;1)}(z^3), & \mbox{if} & i=2,\\
g_2(z^3), & \mbox{if} & i=3,\\
\tilde{f}_{(0\;1)}(z^3)\tilde{f}_{(0\;2)}(z^3), & \mbox{if} & i=4,\\
g_{i-2}(z^3), & \mbox{if} & 5\leq i\leq 7,
\end{array} \right.$
 \end{center}
 then the normal form of the canonical embedding $f_c$ is given as 
 $$f_{c}(z)=[1:z:z^2:h_1(z):\cdots  :h_{7}(z)].$$

 Additionally, we obtain 
 $$a_i=i, \; 1\leq i\leq 5, \; a_6=7,\; a_7=8,\; a_8=10,\; a_9 = 13, \; a_{10} \geq 16,$$
so $\hat{w}(p)=14\leq w(p)$. The equality is fulfilled when $\pi(p)$ is not a hyperosculating point of the embedding $g:C^3\rightarrow {\mathbb{P}}^5$.  The proof, for this situation, now follows from Lemma \ref{le:open-zariski}.

\smallskip
\noindent
{\bf Case 2:}
 In the rest of we assume that $(n-1)(k-1)-2\geq k$, and we define the functions
$$h_{(i\;j)}(z):=z^jg_{(i\;j)}(z^k),\; 0\leq j\leq k-1,\; 1\leq i \leq t_{j}.$$

 Considering an automorphism of ${\mathbb{P}}^{g-1}$, where $g$ is the genus of $S$, we can 
 suppose that the canonical embedding $f$ around the point $p$ is
$$\tiny f(z):= [1:\cdots:z^{k-1}: h_{(1\;0)}(z):\cdots:h_{(1\;k-1)}(z):h_{(2\;0)}(z):\cdots:
h_{(r\;k-2)}(z):$$$$:h_{(r+1\;0)}(z): \cdots:h_{(t_0\;0)}(z)],$$
where $r=t_{k-1}+1$.

Let $a_i$, $1\leq i\leq$ be  the gap values of  $p$. Since $ki+j\leq kl_i+j={\mathrm{Ord}} h_{(i\;j)}(z)$, 
$1\leq j \leq k-1$ and $1\leq i\leq t_j$, the following inequality is obtained
$$\sum_{j=0}^{k-1}\sum_{i=0}^{t_{j}}(ki+j+1)\leq \sum_{i=1}^{g}a_{i}.$$

By Proposition \ref{pr:dim-s}, we have
$$\sum_{j=0}^{k-1}\sum_{i=0}^{t_{j}}(ki+j+1)-\frac{g(g+1)}{2}=\frac{1}{24}
(k-1) (k^{n-1}-2) (k^{n} + k^{n-1}-12),$$
from which we obtain the inequality part of the theorem, as the weight of the point $p$ is
$$w(p):=\sum_{i=1}^{g}a_i-\frac{g(g+1)}{2}.$$

\begin{remark}
\label{re:weie-weight}
Observe that, for each $0\leq j\leq k-1$,  $l_i=i$, for each $1\leq i\leq t_j$, if and only if  $w(p)=\hat{w}(p)$.  Combining this with 
Remark \ref{re:li=i} we obtain that $w(p)=\hat{w}(p)$ if and only if  $\pi(p)$ is not a hyperosculating point of the morphism $\C{k}{n-3}{\lambda}\rightarrow {\mathbb{P}}(Q(r-j))$
for all  $0\leq j\leq k-1$.
\end{remark}

As the  set $\mho$ of the Lemma \ref{le:open-zariski} satisfies the condition of the previous remark, this finishes the proof of the last part of the theorem.

\subsection{On the Example \ref{ej:bound}}\label{Sec:finalejemplo}
Now we will verify that the bound of Theorem \ref{th:w-peso} is not achieved in Example \ref{ej:bound}.
Consider the morphism 
$\pi: C^5(-1)\rightarrow C^5\; :[x_0:x_1:x_2:x_3]\rightarrow [x_0:x_1:x_2].$
As seen  in the introduction $[1:1:-\sqrt[5]{2}]$ is a  Weierstrass point of $C^5$, and in particular is a hyperosculating point of $C^5\rightarrow {\mathbb{P}}(Q(6))$, 
(this morphism is the canonical embedding).  By virtue of Remark \ref{re:weie-weight} we obtain that $\hat{w}(p)<w(p)$.

\section*{Acknowledgments} 
The authors are very grateful to the referee whose valuable suggestions and comments helped to improve the content and clarity in the presentation of this paper.


\end{document}